\documentclass[11pt]{amsart}
\usepackage[utf8]{inputenc}
\usepackage{upref, todonotes}

\numberwithin{equation}{section}
\newtheorem{theorem}{Theorem}[section]
\newtheorem{lemma}[theorem]{Lemma}
\newtheorem{proposition}[theorem]{Proposition}
\theoremstyle{remark}
\newtheorem{definition}[theorem]{Definition}

\newtheorem{remark}[theorem]{Remark}

\DeclareMathOperator{\diam}{diam}
\DeclareMathOperator{\dist}{dist}
\DeclareMathOperator{\cp}{cap}
\title{Continuity  of logarithmic  capacity} 

\author{Sergei Kalmykov}
\address{School of mathematical sciences, Shanghai Jiao Tong University, 800 Dongchuan RD, Shanghai 200240, China \\and
\\
Keldysh Institute of Applied Mathematics  of Russian Academy of Sciences, Miusskaya pl., 4,
125047, Moscow, Russia} 
\email{kalmykovsergei@sjtu.edu.cn}
\thanks{First author supported  by Moscow Center for Fundamental and Applied Mathematics, Agreement with the Ministry of Science and Higher Education of the Russian Federation, No. 075-15-2019-1623 and by SJTU Start-up Grant Program.}

\author{Leonid V. Kovalev}
\address{215 Carnegie, Mathematics Department, Syracuse University, Syracuse, NY 13244, USA}
\email{lvkovale@syr.edu}
\thanks{Second author supported by the National Science Foundation grant DMS-1764266.}

\subjclass[2010]{Primary 31A15; Secondary 31A05, 31A25}
\keywords{Logarithmic capacity, Green's function, uniformly perfect sets, Hausdorff metric}

\begin{document}
\baselineskip6.5mm

\begin{abstract}
We prove the continuity of logarithmic capacity under Hausdorff convergence of uniformly perfect planar sets. The continuity holds when the Hausdorff distance to the limit set tends to zero at sufficiently rapid rate, compared to the decay of the parameters  involved in the uniformly perfect condition. The continuity may fail otherwise.  
\end{abstract}

\maketitle

\section{Introduction}

The studies of the continuity of set capacity and related quantities have a long history in potential theory. In 1961, Gehring proved that the conformal modulus of planar annuli is continuous under Hausdorff convergence of the boundary components~\cite{Gehring1961}. He then extended this results to the modulus of rings in space~\cite{Gehring1962}. 
Aseev~\cite[Theorem 7]{Aseev} proved the continuity of condenser capacity under the Hausdorff convergence of its plates, under the assumption that the plates are uniformly perfect with the same constant $\alpha$. Aseev and Lazareva~\cite{AseevLazareva} proved the analogous continuity result for logarithmic capacity of sets. Ransford, Younsi and Ai~\cite{RansfordYounsi} recently proved that logarithmic capacity of a set varies continuously under holomorphic motions.  

A more general theorem of Aseev~\cite[Theorem 6]{Aseev} involves the concept of strong convergence. According to~\cite{Aseev}, a sequence of sets $E_n$ strongly converges to a set $E$ if there exists $\alpha>0$ such that $E_n$ can be expressed as the union of $\alpha$-uniformly perfect sets with diameters bounded below by some constant $\delta_\alpha(E_n)$, and $d(E_n, E)/\delta_\alpha(E_n)\to 0$. By~\cite[Theorem 6]{Aseev} the conformal capacities $\cp(E^0_n, E^1_n)$ converge to $\cp (E^0, E^1)$ if $E^j_n\to E^j$ strongly for $j=0,1$. 

Although Aseev's theorem weakens the assumption of the sets being uniformly perfect, it does not cover the naturally occurring case of uniformly perfect sets with parameters that are not bounded away from $0$. This is the setting of the present article. Our main result is Theorem~\ref{Green-convergence-thm}, which asserts in part that $\cp(E_n)\to \cp(E)$ whenever $E_n$ is $\alpha_n$-uniformly perfect and the Hausdorff distance $d_n := d_H(E_n, E)$ 
tends to $0$ sufficiently quickly, compared to $\alpha_n$ (these quantities are introduced in Definition~\ref{def-up-set}, \eqref{eq:capacity}, and \eqref{eq:Haus_dist}). Specifically, having 
\[
\log \frac{1}{\alpha_n} \le \frac{1}{24} \frac{\log (1/d_n)}{\log \log (1/d_n)}
\]
is sufficient for continuity, by Remark~\ref{rem:loglog}. 

On the other hand, Proposition~\ref{example-arcs} shows that an inequality of the form 
\[
\log \frac{1}{\alpha_n} \le \frac{C}{d_n}
\]
does not ensure that $\cp E_n$ converges to $\cp E$.  In the final section we consider an application of our main result to the NED property~\cite{AhlforsBeurling} of Cantor-type sets. 

\section{Definitions and preliminary results}

A domain $\Omega$ in the extended complex plane $\overline{\mathbb C}$ is  regular for the Dirichlet boundary problem if for every $w\in \Omega$ there exists \textit{Green's function} $g(\cdot, w, \Omega)$ such that 
\begin{itemize}
\item $g(\cdot, w, \Omega)$ is continuous on $\overline{\mathbb C}\setminus \{w\}$ and is zero on $\overline{\mathbb C}\setminus \Omega$;
\item $g(\cdot, w, \Omega)$ is harmonic in $\Omega\setminus \{w\}$;
\item $g(z,w, \Omega) = -\log|z-w| + O(1)$ as $z\to w$ if $w$ is finite;
\item $g(z,w, \Omega)=\log|z|+O(1)$, as $z\rightarrow w$ if  $w=\infty$.
\end{itemize}

For a general domain with non-polar complement, Green's function still exists if the continuity requirement is relaxed to allow an exceptional polar subset of $\partial \Omega$~\cite[Definition 4.4.1]{Ransford}. 

When $E$ is a compact subset of $\mathbb C$, we write 
\[
g_E(z) = g(z, \infty, \overline{\mathbb C} \setminus E).  
\]

Let $\nu$ be a finite positive Borel measure of compact support. Its \textit{logarithmic potential} is defined by
\begin{equation}\label{eq:potential}
  U^{\nu}(z):=\int \log \frac{1}{|z-t|} d\nu(t).   
\end{equation}

Let $E\subset \mathbb{C}$ 
be a  compact set in the complex plane.  The collection of all positive unit Borel measures with support in $E$ is denoted $\mathcal{M}(E)$. The \textit{logarithmic energy} of a measure $\nu\in \mathcal{M}(E)$ is defined as 
\begin{equation}\label{eq:log_energy}
    I(\nu):=\iint \log \frac{1}{|z-t|}d\nu(z) d\nu(t), 
\end{equation}
and the energy $V$ of $E$ by
\begin{equation}\label{eq:energy}
    V:=\inf\{I(\nu) \colon  \nu \in \mathcal{M}(E)\}.
\end{equation}
The energy $V$ takes values in $(0, \infty]$. When it is finite, there is a unique \textit{equilibrium measure} $\nu_{E}\in \mathcal{M}(E)$ for which the infimum defining $V$ in~\eqref{eq:energy} is attained. The quantity 
\begin{equation}\label{eq:capacity}
    \cp(E):=e^{-V}
\end{equation}
is called the \textit{logarithmic capacity} of compact set $E$. For a general set $E\subset \mathbb C$, 
\begin{equation}\label{eq:inner-capacity}
    \cp(E):=\sup\{\cp(K)\colon K \subset E, K \text{ is compact}\}
\end{equation}
which is also known as the \emph{inner capacity} of $E$.

Let  $\Omega$ be the \textit{outer domain} relative to $E$, that is the unbounded component of the complement $\overline{\mathbb{C}}\setminus \Sigma$. Then we have an asymptotic expansion (e.g., \cite[p.~53]{SaffTotik})
\begin{equation}\label{eq:Green_pot_cap}
    g(z,\infty,\Omega) = - U^{\nu_{E}}(z)+\log\frac{1}{\cp(E)}.
\end{equation}

Our approach requires an explicit H\"older estimate for Green's function of a uniformly perfect set. 

\begin{definition}\label{def-up-set}
A closed set $E\subset \mathbb C$ is \emph{uniformly perfect} if there exists a constant $\alpha\in (0, 1)$ such that the set $E\cap \{z\colon \alpha r\le |z-a|\le r\}$ is nonempty for every $a\in E$ and every $r$ such that $0<r\le \diam E$. 
\end{definition}
To emphasize the value of $\alpha$, we sometimes call $E$ an $\alpha$-uniformly perfect set.

By \cite[Theorem 1]{Pommerenke79}, 
\begin{equation}\label{capacity-lower-bound}
\cp(E\cap \overline{B}(a, r)) \ge \frac{\alpha^2}{32} r    
\end{equation}
for all $0<r\le \diam E$. Note that the definition of a uniformly perfect set in ~\cite{Pommerenke79} requires the set to be unbounded. However, the proof of \cite[Theorem 1]{Pommerenke79} applies verbatim to our situation. 

The following theorem of Siciak~\cite[Theorem 4.1]{Siciak} provides a H\"older estimate for Green's function of a uniformly perfect set. It requires additional notation. Given a compact subset $E$ of a disk $B(a, R)$,  let $h(\cdot, E, B(a, R))$ be the Perron solution~\cite[Def. 4.1.1]{Ransford} 
of the Dirichlet problem $\Delta u=0$ in $B(a, R)\setminus E$, $u=0$ on $E$, and $u=1$ on $\partial B(a, R)$. For $1\le r<R$ let
\[
c(E; B(a, R), \overline{B}(a, r)) = 1 - \sup_{|z-a|=r} h(z, E, B(a, R)).
\]
This capacity-like quantity takes values between $0$ and $1$ and is monotone with respect to $E$. Finally, let 
\[
c_E(a, t, r, R) = c(E\cap \overline{B}(a, t); B(a, tR), \overline{B}(a, tr)), \quad 0\le t\le 1
\]
which represents the $c$-capacity
of the $t$-neighborhood of $a$ in $E$, scaled according to its size. 

\begin{theorem}\label{Siciak-quote}~\cite[Theorem 4.1]{Siciak} 
Let $1\le r<R<\infty$ and let $\{\rho_n\}$ be a sequence of real numbers such that $0<\rho_n<1$ and 
\begin{equation}\label{Siciak-quote1}
    \frac{R}{r} \le \frac{\rho_n}{\rho_{n+1}} \le B<\infty,\quad n\ge 1.
\end{equation}
If $a$ is a point of a compact set $E$ 
of $\mathbb C$ such that $c_E(a, \rho_n, r, R)\ge m>0$ ($n\ge 1$), then for every $\rho>0$ the function $g_{E\cap \overline{B}(a, \rho)}$ is H\"older continuous at $a$ with exponent $\mu=m/\log B$:
\begin{equation}\label{Siciak-quote2}
g_{E\cap \overline{B(a, \rho)}}(z) \le M \delta^{m/\log B}, \quad |z-a|\le \delta\le 1, 
\end{equation}
where $M$ depends only on $\rho, r, R, m$, and $B$. 
\end{theorem}

\begin{definition}\label{def: Haus_dist}
The \textit{Hausdorff distance} between two nonempty bounded closed sets $A$ and $B$ is defined as
\begin{equation}\label{eq:Haus_dist}
    d_H(A,B) := \max \left(\sup_{z\in A}\inf_{t\in B}|z-t|, \sup_{z\in B}\inf_{t\in A}|z-t|\right). 
\end{equation}
\end{definition}

\begin{definition}
Let $w\in\mathbb{C}$ be given and let $\Omega_n\subset\overline{\mathbb C}$ be domains such that $w$ is an interior point of $\bigcap_{n=1}^\infty \Omega_n$. 
Following~\cite[p. 13]{Pommerenke}, we say that 
\[
\Omega_n \rightarrow \Omega \ \ \text{as} \ \ n\rightarrow \infty \ \ \text{with respect to } w 
\]
in the sense of {\it kernel convergence} if 
\begin{itemize}
\item $\Omega$ is a domain such that $w\in \Omega$ and some neighborhood of every $z\in \Omega$ lies in $\Omega_n$ for large $n$;
\item for $z\in \partial \Omega$ there exist $z_n\in \partial \Omega_n$ such that $z_n\rightarrow z$ as $n\rightarrow \infty$. 
\end{itemize}
\end{definition}

An equivalent definition is found in ~\cite[p.~77]{Duren} and \cite[p. 54]{MR0247039}. According to it, $\Omega$ is the kernel of $\{\Omega_n\}$ if it is the maximal domain containing $w$ such that every compact subset of $\Omega$ belongs to all but finitely many of the domains $\Omega_n$. The convergence to $\Omega$ in the sense of kernel requires that every subsequence of $\{\Omega_n\}$ also has the same kernel $\Omega$.  

\section{Main results}

Our first step is to prove a version of the estimate~\eqref{Siciak-quote2} for Green's function in which the modulus of continuity has an explicit value of the multiplicative constant $M$. Such an estimate will be obtained from the formula (2a) in the proof of ~\cite[Theorem 4.1]{Siciak}, which states that under the assumptions of Theorem~\ref{Siciak-quote}, 
\begin{equation}\label{Siciak-quote3}
    h(z, E\cap \overline{B}(a, \rho_n), B(a, R\rho_n)) \le \left( \frac{1}{r\rho_n}\right)^{m/\log B} \delta^{m/\log B}
\end{equation}
for all $z$ with $|z-a|\le \delta \le \min(1, r\rho_{n+1})$. 

In order to use~\eqref{Siciak-quote3}, we need to relate the function $h$ to Green's function $g_E$. 

\begin{lemma}\label{h-g-comparison-lem} Suppose that $E\subset \mathbb C$ is a compact set of positive logarithmic capacity. Then for $a\in E $ and for $R>\diam E$ we have 
\begin{equation}\label{h-g-comparison}
    g_E(z) \le 
    \left(\log\frac{R+\diam E}{\cp E} \right) 
    h(z, E, B(a, R)) 
\end{equation}
for all $z$ with $|z-a|\le R$. 
\end{lemma}

\begin{proof} Both sides of~\eqref{h-g-comparison} are harmonic in the set $\Omega = B(a, R)\setminus E$ and vanish on $E$ up to a polar set. By definition, $h=1$ on $\partial B(a, R)$. Writing $g_E$ in terms of the potential of the  equilibrium measure $\nu_E$, we obtain from~\eqref{eq:Green_pot_cap} that for all $z\in \partial B(a, R)$, 
\[
g_E(z) = \log\frac{1}{\cp E} + \int \log |z-\zeta|\,d\nu(\zeta) \le 
\log\frac{1}{\cp E} +  \log (R+\diam E)
\]
because $|z-\zeta|\le R+\diam E$ and $\nu_E$ is a probability measure. Hence~\eqref{h-g-comparison} holds on $\partial B(a, R)$. The maximum principle completes the proof. 
\end{proof}
 
We also need to translate the  geometric property of being $\alpha$-uniformly perfect into a lower bound on the capacity  $c_E(a, \rho)$ required by Theorem~\ref{Siciak-quote}. Lemma 1.7 of~\cite{Siciak} states that for $0 < t\le 1$,
\begin{equation}\label{c-bound}
    c_E(a, t, r, R) \ge \left(\log \frac{R-1}{r+1} \right) 
    \left( 
\log \frac{t(R-1)}{\cp(E\cap \overline{B}(a, t))} \right)^{-1}
\end{equation}
provided that $r<R-2$. Combining~\eqref{capacity-lower-bound} and ~\eqref{c-bound} we obtain the following: if $E\subset \mathbb C$ is an $\alpha$-uniformly perfect set of diameter $1$, then for every $a\in E$ and $0< t\le 1$, 
\begin{equation}\label{c-bound2}
    c_E(a,t,r,R) \ge \left(\log \frac{R-1}{r+1} \right) 
    \left( 
\log \frac{32(R-1)}{\alpha^2}\right)^{-1}
\end{equation}
provided that $r<R-2$.

\begin{theorem}\label{Siciak-bound}
Let $E$ be an $\alpha$-uniformly perfect set such that $0\in E$ and $\diam E < \infty$. Let 
\begin{equation}\label{Siciak-params}
\begin{split}
M & = M(\alpha) =    \log \frac{384}{\alpha^2}, \\ 
\beta & = \beta(\alpha) = \frac{ \log(9/8)}{2\log 2}
    \left( 
\log \frac{288}{\alpha^2}\right)^{-1}.
\end{split}
\end{equation}
Then for $|z|\le \diam E$ we have
\begin{equation}\label{Siciak-bound-eqn}
g_E(z) \le M|z|^{\beta}/(\diam E)^\beta. 
\end{equation}
\end{theorem}

\begin{proof} The problem reduces to the case $\diam E=1$ by rescaling.

Let $a=0$, $R=10$, and $r=7$ in~\eqref{c-bound2}: 
\begin{equation}\label{c-bound3}
    c_E(a, t, 7, 10) \ge m:= \left(\log \frac{9}{8} \right) 
    \left( 
\log \frac{288}{\alpha^2}\right)^{-1} ,\quad 0<t\le 1.
\end{equation}
In Theorem~\ref{Siciak-quote} choose $B=2$ and $\rho_n = 2^{-n}$. Then the potential function $h$ associated with the set $E_n = E\cap \overline{B}(0, 2^{-n})$ can be estimated by~\eqref{Siciak-quote3} as follows.  
\begin{equation}\label{h-bound}
    h(z, E_n, B(0, 10\cdot 2^{-n})) \le \left( \frac{2^n}{7}\right)^{m/\log 2} \delta^{m/\log 2}
\end{equation}
for all $z$ with $|z|\le \delta \le \min(1, 7\cdot 2^{-n-1})$. Note that $\diam E_n\le 2^{1-n}$ and $\cp E\ge \frac{\alpha^2}{32} 2^{-n}$ by~\eqref{capacity-lower-bound}. 
Apply Lemma~\ref{h-g-comparison-lem} to $E_n$ with $a=0$ and $R=10\cdot 2^{-n}$, and then invoke~\ref{h-bound} to obtain
\begin{equation}\label{g-h-explicit}
\begin{split}
    g_{E_n}(z) 
& \le  \left(\log \frac{384}{\alpha^2}\right) h(z, E_n, B(0, 10\cdot 2^{-n})) \\ 
& \le \left(\log \frac{384}{\alpha^2}\right)  \left( \frac{2^n}{7}\right)^{m/\log 2} \delta^{m/\log 2}
\end{split}
\end{equation}
for all $z$ with $|z|\le \delta \le \min(1, 7\cdot 2^{-n-1})$. 

Given a complex number $z$ with $|z|\le 1$, let 
$n$ be the smallest integer such that $|z|\le 4^{1-n}$. Since $4^{1-n}<7\cdot 2^{-n-1}$, we can choose $\delta=4^{1-n}$ in~\eqref{g-h-explicit}, thus obtaining  
\begin{equation}\label{g-explicit}
\begin{split}
    g_{E_n}(z) 
& \le \left(\log \frac{384}{\alpha^2}\right)  \left( \frac{4}{7} 2^{-n}\right)^{m/\log 2}  \\ 
& \le   \left(\log \frac{384}{\alpha^2}\right)    |z|^{m/(2\log 2)}
\end{split}
\end{equation}
where the last inequality holds because $4^{-n} < |z|$ by the choice of $n$. This proves~\eqref{Siciak-bound-eqn}. 
\end{proof}

We are ready to state and prove our main result.

\begin{theorem}\label{Green-convergence-thm}
Suppose that for each $n\in \mathbb N$, $E_n$ is a compact $\alpha_n$-uniformly perfect subset of $\mathbb C$. Furthermore, suppose $E_n\to E\subset \mathbb C$ in the Hausdorff metric $d_H$, where $E$ is a compact set with more than one point. 
If  the sequence
\begin{equation}\label{rate-dh}
d_H(E_n, E)  \exp\left(24 
\log \frac{1}{\alpha_n} \log\log \frac{1}{\alpha_n} \right)    
\end{equation}
is bounded, then Green's functions $g_{E_n}(z)$ converge to $g_E(z)$ uniformly with respect to $z\in \mathbb C$. As a consequence, 
$\cp(E_n)\to \cp(E)$ and the unbounded component $\Omega$ of $\mathbb C\setminus E$ is a regular domain for the Dirichlet problem. 
\end{theorem}

The proof requires a lemma from~\cite{Kalmykov_Kovalev} which applies to our situation, because Hausdorff convergence of the complements implies kernel convergence (see e.g. \cite[p.~54]{MR0247039}).

\begin{lemma}\label{lem:oneside}~\cite[Lemma 4.1]{Kalmykov_Kovalev} If  $\Omega_n\to \Omega$ in the sense of kernel with respect to $w\in \Omega$, then  
\[
\sup_{z\ne w} (g(z,w,\Omega) - g(z,w,\Omega_n)) \to 0.
\]
\end{lemma}

\begin{proof}[Proof of Theorem~\ref{Green-convergence-thm}] 
The main goal is to prove uniform convergence of Green's functions. The convergence of logarithmic capacity follows from it because of the asymptotic expansion 
\[g_E(z) = \log|z| - \log\cp(E) + o(1),\quad z\to\infty. \]
Also, uniform convergence $g_{E_n}\to g_E$ allows us to interchange limits with respect to $n$ and $z$ below: for every $\zeta\in \partial \Omega $
\begin{equation}\label{exchange}
\lim_{z\to\zeta} g_E(z) 
= \lim_{z\to\zeta}\lim_{n\to\infty} g_{E_n}(z)
= \lim_{n\to\infty} \lim_{z\to\zeta} g_{E_n}(z) = 0.
\end{equation}
By ~\cite[Theorem 4.4.9]{Ransford}, if for some $w\in \Omega$ Green's function $g(\cdot, w, \Omega)$ vanishes on $\partial\Omega$, then $\Omega$ is a regular domain for the Dirichlet problem. Thus, ~\eqref{exchange} shows that the unbounded component of $\overline{\mathbb C}\setminus E$ is regular. 

We proceed to prove that $g_{E_n}\to g_E$ uniformly.
By rescaling, it suffices to consider the case $\diam E=1$. The assumption $d_H(E_n, E)\to 0$ implies  $\diam E_n\to 1$. Let $M_n = M(\alpha_n)$ and $\beta_n = \beta(\alpha_n)$ be as in Theorem~\ref{Siciak-bound}. We claim that
\begin{equation}\label{eq:rate-dh}
 M_n d_H(E_n, E)^{\beta_n} \to 0. 
\end{equation}
Indeed, if $\alpha_n$ is bounded from below by a positive constant $\alpha$, then $M_n\le M(\alpha)$ and $\beta_n\ge M(\alpha)$, hence~\eqref{eq:rate-dh} holds by virtue of  $d_H(E_n, E)\to 0$. Consider the case $\alpha_n\to 0$. The logarithm of the left-hand side of~\eqref{eq:rate-dh} does not exceed 
\begin{equation}\label{eq:rate-dh2}
A + \log \log \frac{384}{\alpha_n^2} - 
C \log \frac{1}{\alpha_n} \log\log \frac{1}{\alpha_n}  \left(\log \frac{288}{\alpha_n^2} \right)^{-1}
\end{equation}
where  $A$ is some constant and 
\[C = 24 \frac{\log(9/8)}{2\log 2} > 2.\]
Up to a bounded additive term, the  expression~\eqref{eq:rate-dh2} simplifies to
\[
(1-C/2) \log\log \frac{1}{\alpha_n} \to -\infty
\]
which proves~\eqref{eq:rate-dh} in this case as well. The case when $\limsup \alpha_n > \liminf\alpha_n=0$ follows by considering subsequences. 

Let $g=g_{E}$ and $g_n=g_{E_n}$. 
Suppose that uniform convergence $g_n\to g$ fails. Using Lemma~\ref{lem:oneside} and passing to a subsequence, we may assume there exist $\epsilon>0$ such that $\sup_{\mathbb C}(g_n-g)\ge \epsilon$ for all $n$. The function $g_n-g$ is bounded and harmonic on the set $U = \overline{\mathbb C}\setminus (\partial E_n\cup \partial E)$ after the removable singularity at $\infty$ is eliminated. Since $g_n-g \le 0$ on $\partial E_n$, it follows from the maximum principle that $g_n-g$ attains its maximum on $\partial E$.  Pick $z_n\in\partial E$ such that $g_n(z_n)\ge \epsilon$.  

By Theorem~\ref{Siciak-bound} and~\eqref{eq:rate-dh}, 
\[
g_n(z_n) \le M_n \left(\frac{\dist(z_n, E_n)}{\diam E_n}\right)^{\beta_n} 
\le M_n \left(\frac{d_H(E, E_n)}{\diam E_n}\right)^{\beta_n} \to 0 
\]
contradicting the choice of $z_n$. This contradiction concludes the proof.
\end{proof}

\begin{remark}\label{rem:loglog} The boundedness assumption in Theorem~\ref{Green-convergence-thm} holds if for all sufficiently large $n$,
\begin{equation}\label{eq:loglog}
\log\frac{1}{\alpha_n} \le 
\frac{1}{24} \frac{\log b_n}{\log \log b_n}
\end{equation}
where $b_n = 1/d_H(E, E_n) \to \infty$ as $n\to\infty$. 
\end{remark}

Indeed, for large $n$ we have $\log \log b_n > 1$, hence
~\eqref{eq:loglog} implies 
\[
24 \log\frac{1}{\alpha_n} \log\log \frac{1}{\alpha_n}  
\le 
\frac{\log b_n}{\log \log b_n}
(\log \log b_n  - \log 24) 
< \log b_n
\]
Therefore the sequence~\eqref{rate-dh} is bounded by $1$. 

\section{Examples and applications}\label{sec:examples}

If a sequence of $\alpha_n$-uniformly perfect sets $E_n$ has $\alpha_n\to 0$ much faster than $d_H(E_n, E)\to 0$, the logarithmic capacity of $E_n$ may fail to converge to the logarithmic capacity of $E$. The following proposition presents a concrete form of this observation.

\begin{proposition}\label{example-arcs} In Theorem~\ref{Green-convergence-thm}, the sequence~\eqref{rate-dh} cannot be replaced by $d_H(E_n, E)\log \alpha_n$. More precisely, there exists a sequence of compact sets $E_n$ which are $\alpha_n$ - uniformly perfect and converge to $\mathbb T$ in the uniform metric in such a way that $d_H(E_n, E)\log \alpha_n$ is bounded, yet the uniform convergence of Green's functions fails.
\end{proposition}

\begin{proof} The idea of this example goes back to Ahlfors and Beurling~\cite[Theorem 17]{AhlforsBeurling}. 
Given a sequence of numbers $L_n\in (0, \pi)$, we construct a sequence of compact subsets of the unit circle $\mathbb T$ as follows: 
\[
E_n = \{z\in \mathbb T\colon |\arg(z^n)|\le L_n \} 
\]
where $\arg$ is the principal branch of the argument, taking values between $-\pi$ and $\pi$. The set $E_n$ consists of $n$ uniformly distributed arcs of length $2L_n/n$. The gaps between these arcs have length $2(\pi-L_n)/n$, which implies that 
\begin{equation}\label{ex:dh}
d_H(E_n, \mathbb T) = 2\sin \frac{\pi-L_n}{2n} 
\end{equation}
Hence $E_n\to \mathbb T$ in the Hausdorff metric. 

The diameter of each connected component of $E_n$ is $2\sin (L_n/n)$ and the distance from a component to the rest of $E_n$ is $2\sin((\pi-L_n)/n)$. Suppose that an annulus $\{z\colon r<|z-a|<R\}$ separates $E_n$. Since the disk $\{z\colon |z-a|\le r\}$ contains a connected component of $E_n$, we have $r\ge \sin (L_n/n)$. Since also $R-r\le 2\sin((\pi-L_n)/n)$, it follows that
\[
\frac{R}{r} \le 1 + \frac{2\sin((\pi-L_n)/n)}{\sin (L_n/n)}
\]
Hence, $E_n$ is $\alpha_n$-uniformly perfect with 
\[
\alpha_n \ge \left(1 + \frac{2\sin((\pi-L_n)/n)}{\sin (L_n/n)}\right)^{-1}
\]
If $L_n\to 0$, this bound on $\alpha_n$ is asymptotic to $L_n/(2\pi)$. 

The logarithmic capacity of the circular arc $\Gamma_n = \{e^{it}\colon |t|\le L_n\}$ is equal to $\sin(L_n/2)$ (see e.g.~\cite[Ch.~5, Table~5.1, p.~135]{Ransford}). 
Since the set $E_n$ is the preimage of $\Gamma_n$ under the polynomial $z\mapsto z^n$, it  follows that (\cite[Theorem~5.2.5, p.~134]{Ransford})  
\[
\cp E_n = (\cp \Gamma_n)^{1/n} = (\sin(L_n/2))^{1/n}
\]
Thus, $\cp E_n\to \cp \mathbb T=1$ if and only if $\log(1/L_n) = o(n)$ as $n\to \infty$. 

For example, the choice $L_n = \exp(-n)$ results in $\cp E_n\not\to \cp \mathbb T$, which also indicates the failure of uniform convergence of Green's functions. With this choice we have $\log 1/\alpha_n$ asymptotic to $n$ and $d_H(E_n, \mathbb T)\le \pi/n$ by virtue of~\eqref{ex:dh}. Thus the product $d_H(E_n, \mathbb T)\log \alpha$ is bounded. 
\end{proof}

There remains a substantial gap between the assumptions of Theorem~\ref{Green-convergence-thm} and Proposition~\ref{example-arcs}. 
As an application of Theorem~\ref{Green-convergence-thm} we consider the NED property of Cantor-type sets. The notion of an \emph{NED set} is an important function-theoretic concept of a removability, introduced by Ahlfors and Beurling in~\cite{AhlforsBeurling}. For example, NED sets are removable for holomorphic functions $f$  with finite Dirichlet integral $\int |f'|^2$ and for extremal distances. 
We do not state the general definition of NED sets here, because the following theorem of Ahlfors and Beurling~\cite[Theorem 14]{AhlforsBeurling} suffices for other purposes: a compact subset $K$ of an interval $I$ is NED if and only if 
\begin{equation}\label{inner-capacity-equality}
\cp(I\setminus K) = \cp (I).    
\end{equation}
The left hand side of~\eqref{inner-capacity-equality} is the inner capacity~\eqref{eq:inner-capacity} of the non-compact set $I\setminus K$. 

Let $I=[0, 1]$. Given a sequence of numbers $\epsilon_n \in (0, 1)$, let $K_0=I$ and inductively construct the sets $K_1 \supset K_2\supset \dots$ so that $K_{n}$ is obtained by removing the middle  $\epsilon_n$-part of each connected component of $K_{n-1}$. The intersection $K=\bigcap_{n=0}^\infty K_n$ is a Cantor-type set which becomes the standard middle-third Cantor set if $\epsilon_n=1/3$ for all $n$. 
Let $E_n=\overline{I\setminus K_n}$ for $n=1, 2, \cdots$. It is easy to show that $E_n\to [0, 1]$ in the Hausdorff distance; see the proof of Proposition~\ref{prop:Cantor-convergence} below. By the definition of inner capacity, property~\eqref{inner-capacity-equality} holds if and only if $\cp(E_n) \to \cp(I)$ as $n\to\infty$. This leads us to the following result. 

\begin{theorem}\label{prop:Cantor-convergence}
Suppose $K$ is a Cantor-type set determined by a sequence of numbers $\epsilon_n \in (0, 1)$ such that 
\begin{equation}\label{eq:Cantor-condition}
    \log\frac{1}{\epsilon_n} \le \frac{Cn}{\log n}, \quad n\ge 2,
\end{equation}
for some constant $C<1/(24\log 2)$. Then~\eqref{inner-capacity-equality} holds, and consequently $K$ is an NED set.  
\end{theorem}

\begin{proof} Since $K_n$ consists of $2^n$ disjoint segments of equal length, each of them has length at most $2^{-n}$. Therefore, the $2^{-n-1}$ neighborhood of $E_n$ covers $I$. It follows that $d_H(E_n, I)\le 2^{-n-1}$.

We claim that the set $E_n$  is $\alpha_n$-uniformly perfect where $\alpha_n=\frac12 \min_{k\le n} \epsilon_k$. Since $E_1$ is an interval, it suffices to consider $n\ge 2$.  Note that the set $E_k$ is constructed by inserting an interval in the middle of each component of $[0, 1]\setminus E_{k-1}$; the length of this interval is $\epsilon_k \ell$ where $\ell$ is the length of the component. Therefore, the distance from the inserted interval to $E_{k-1}$ is $(1-\epsilon_k)\ell/2$. It follows that every connected component $J$ of the set $E_n$ satisfies 
\begin{equation}\label{distance-component}
 \dist(J, E_n\setminus J) \le \frac{1-\alpha_n}{2 \alpha_n}    \diam J.  
\end{equation}
Suppose that $a\in E_n$, $0<r\le \diam E$, and the annulus $\{z\colon \alpha r\le |z-a|\le r\}$ is disjoint from $E$. Let $k$ be the smallest index such that $E_k \cap B(a, \alpha r)$ is nonempty. If $k=1$, then $B(a, \alpha r)$ contains $[(1-\epsilon_1)/2, (1+\epsilon_1)/2]$, hence $\alpha r\ge \epsilon_1/2$. And since $r\le \diam E_n \le 1$, it follows that $\alpha \ge \epsilon_1 /2\ge \alpha_n$ as claimed. 

Suppose $k\ge 2$. If $B(a, \alpha r)$ contained more than one component of $E_k$, then it would also contain a component of $E_{k-1}$ situated between those, contrary to the choice of $k$. Thus, the set $J = E_k\cap B(a, \alpha r)$ is connected.  Since $\diam J \le 2\alpha_n r$, the estimate   ~\eqref{distance-component} implies   
\[
(1-\alpha) r \le 
 \dist(J, E_k\setminus B(a, \alpha r)) 
 = \dist(J, E_k\setminus J) \le (1-\alpha_n)r, 
\]
hence $\alpha \ge \alpha_n$. This completes the proof that $E_n$ is $\alpha_n$-perfect.  

To justify the application of Theorem~\ref{Green-convergence-thm}, we use Remark~\ref{rem:loglog}. Indeed, in the inequality \eqref{eq:loglog} we have $\log b_n = \log(1/d_H(E_n, I)) \ge (n+1)\log 2$, which in view of ~\eqref{eq:Cantor-condition} implies that ~\eqref{eq:loglog} holds. Thus, $\cp(E_n)\to \cp(I)$. \end{proof}

\bibliographystyle{amsplain}
\bibliography{amsrefs.bib}

\end{document}